\documentclass[12pt]{amsart}
\usepackage{amsmath}
\usepackage{amsfonts}
\usepackage{amssymb}
\usepackage{amsthm}
\usepackage{verbatim} 

\usepackage{enumerate}
\usepackage{tikz}
\usepackage{color, comment}
\usepackage[colorlinks,citecolor=blue,backref=page, linkcolor=blue]{hyperref}

\usepackage[utf8]{inputenc}
\usepackage[square,sort,comma,numbers]{natbib}
\usepackage{esint}
\usepackage{bbm}
\usepackage{cancel}
\usepackage{subcaption}
\usepackage{listofitems}
\usepackage{pgfplots}
\usepackage{mathrsfs}
\usetikzlibrary{arrows}
\usepackage[a4paper,left=35mm,right=35mm,top=35mm,bottom=35mm,marginpar=25mm]{geometry}

\usepackage{import}
\usepackage{xifthen}
\usepackage{pdfpages}
\usepackage{transparent}

\definecolor{ccqqqq}{rgb}{0.8,0.,0.}
\definecolor{qqqqff}{rgb}{0.,0.,1.}
\definecolor{xfqqff}{rgb}{0.4980392156862745,0.,1.}

\title[Intrinsic Timed Hausdorff and Implications]{Intrinsic Timed-Hausdorff Convergence and Its Implications}

\author[Perales]{Raquel Perales}
\thanks{R. Perales was funded by 
the Austrian Science Fund (FWF) [Grant DOI: 10.55776/EFP6]}
\address[R. Perales]{Math Dept. University of Vienna, Oskar-Morgenstern-Platz 1, 1090 Wien. Centro de Investigaci\'on en Matem\'aticas, 
De Jalisco s/n, Valenciana, Guanajuato, Gto. Mexico. 36023}
\email{raquel.perales@cimat.mx}

\thanks{}

\theoremstyle{plain}

\newtheorem{thm}{Theorem}[section]
\newcommand{\bt}{\begin{thm}}
\newcommand{\et}{\end{thm}}

\newtheorem{cor}[thm]{Corollary}   

\newcommand{\bc}{\begin{cor}}
\newcommand{\ec}{\end{cor}}

\newtheorem{lem}[thm]{Lemma}   

\newcommand{\bl}{\begin{lem}}
\newcommand{\el}{\end{lem}}

\newtheorem{prop}[thm]{Proposition}
\newcommand{\bp}{\begin{prop}}
\newcommand{\ep}{\end{prop}}

\newtheorem{defn}[thm]{Definition}

\newcommand{\ben}{\begin{itemize}}
\newcommand{\een}{\end{itemize}}

\newcommand{\bd}{\begin{defn}}    
\newcommand{\ed}{\end{defn}}

\newtheorem{rmrk}[thm]{Remark}   

\newcommand{\br}{\begin{rmrk}}
\newcommand{\er}{\end{rmrk}}

\newcommand{\be}{\begin{equation}}

\newcommand{\ee}{\end{equation}}

\newcommand{\N}{\mathbb{N}}

\begin{document}

\begin{abstract}  
Sakovich–Sormani introduced several notions of distance between certain classes of Lorentzian manifolds. These distances use the Hausdorff and Gromov–Hausdorff distances and therefore extend naturally to a broader class of spaces. Here we show that, for timed-metric-spaces, intrinsic timed–Hausdorff convergence implies (timeless) Gromov–Hausdorff convergence as well as big bang convergence, among other related implications for future-developed convergence.
\end{abstract}

\maketitle

\section{{\bf{Introduction}}}\label{sec:intro}

Sakovich and Sormani in \cite{SakSor-Notions} introduced several notions of distance between certain classes of Lorentzian manifolds. To do this, they gave a canonical method for converting smooth Lorentzian manifolds satisfying certain hypotheses—known as causally null compactifiable—into compact timed-metric-spaces. This conversion uses the cosmological time function of Anderson–Gallo-
way–Howard \cite{AGH} together with the null distance of Sormani–Vega \cite{SV-null}.

Other approaches to convergence for Lorentzian manifolds were developed by Minguzzi–Suhr \cite{Minguzzi-Suhr-24} and by Mondino–Sämann \cite{Mondino-Saemann-2025}. Minguzzi–Suhr established a compactness theorem for “Lorentzian metric spaces” (which are not metric spaces). In addition, Mondino–Sämann introduced the notion of Lorentzian Gromov–Hausdorff convergence for “Lorentzian pre-length spaces”, using causal diamonds and time separation functions, and proved a precompactness theorem.

In this short note, we prove several relationships between Sormani-Sakovich's distances in the general setting of timed-metric-spaces, c.f. Conjectures 6.1-6.5 in \cite{SakSor-Notions}. No knowledge of Lorentzian geometry is required, but applications to the convergence of Lorentzian manifolds are expected. 
We note that we do not use the compactness and embedding theorem for the intrinsic timed-Hausdorff distance of Che–Perales–Sormani \cite{ChePeSo}, although it could provide alternative proofs of some of the results.

We begin by defining timed-metric-spaces and stating our first result. All remaining relevant definitions can be found in Section \ref{sec:back}.

\begin{defn}\label{defn:timed-metric-space}
A {\bf timed-metric-space}, 
$(X,d,\tau)$, is a metric space, $(X,d)$,
endowed with a $1$-Lipschitz function
$\tau\colon X\to [0,\infty)$,
that we call time function. 
\end{defn}

\begin{thm}[Part of Conjecture 6.1 in \cite{SakSor-Notions}]\label{conj:tK-conv-to-others}
If $(X_j,d_j, \tau_j)$, $j \in \mathbb N$, and $(X_\infty, d_\infty, \tau_\infty)$
are compact timed-metric-spaces
converging in the intrinsic timed-Hausdorff
sense as in Definition \ref{defn:tau-K-dist},
\be 
d_{\tau-H}\Big((X_j,d_j,\tau_j),(X_\infty,d_\infty,\tau_\infty)\Big)\to 0,
\ee
then we have (timeless) Gromov-Hausdorff convergence:
\be\label{eq-timelessGH}
d_{GH}\Big((X_j,d_j),(X_\infty,d_\infty)\Big) \to 0.
\ee
\end{thm}

There are special classes of timed-metric-spaces in which the time functions have explicit formulas. These include big bang spaces and future-developed spaces. For big bangs, Sormani–Vega \cite{SV-BigBang} and Sakovich–Sormani \cite{SakSor-Notions} defined $BB$–$GH$ convergence, while for future-developed spaces Sakovich–Sormani \cite{SakSor-Notions} introduced $FD$–$HH$ convergence, which can be seen as a particular case of 
the GH convergence for metric pairs defined by Che, Galaz-García, Guijarro, and Membrillo Solis \cite{CGGGMS}. We now state some of the relationships between intrinsic timed–Hausdorff convergence and these other two notions of convergence.

\begin{thm}[Part of Conjectures 6.2 and 6.4 in \cite{SakSor-Notions}]\label{conj:tK-conv-to-BB}
Let $(X_j,d_j, \tau_j)$, $j \in \mathbb N$, and $(X_\infty, d_\infty, \tau_\infty)$
be compact timed-metric-spaces.
\medskip

If the  $(X_j,d_j, \tau_j)$, $j \in \mathbb N$,
are big bang as in Definition~\ref{defn:BB} and 
converge in the intrinsic timed-Hausdorff
sense as in Definition \ref{defn:tau-K-dist},
\be 
d_{\tau-H}\Big((X_j,d_j,\tau_j),(X_\infty,d_\infty,\tau_\infty)\Big)\to 0,
\ee
then 
$(X_\infty, d_\infty, \tau_\infty)$ is a big bang and we have convergence in the BB-GH sense as in Definition \ref{defn:BB-GH},
 \be 
d_{BB-GH}\Big((X_j,d_j,\tau_j),(X_\infty,d_\infty,\tau_\infty)\Big)\to 0.
\ee 
Conversely, if all the $(X_j,d_j, \tau_j)$, $j \in \mathbb N \cup \{\infty\}$,
are big bang and converge in BB-GH sense, 
 \be 
d_{BB-GH}\Big((X_j,d_j,\tau_j),(X_\infty,d_\infty,\tau_\infty)\Big)\to 0
\ee 
then they converge in the intrinsic timed-Hausdorff sense,
\be 
d_{\tau-H}\Big((X_j,d_j,\tau_j),(X_\infty,d_\infty,\tau_\infty)\Big)\to 0.
\ee
\end{thm}

We conclude this section with two results on future-developed spaces.
We recall that Conjecture 6.3 in \cite{SakSor-Notions} asks for a sequence of (smooth) future-developed spaces that converges in the intrinsic timed–Hausdorff sense to a big bang space. Thus, one should not expect a sequence of future-developed spaces that converges in the intrinsic timed–Hausdorff sense to also converge in the FD–HH sense. However, convergence does hold in the more general setting of metric pairs.


\begin{thm}\label{conj:tK-conv-FDm}
Let $(X_j,d_j, \tau_j)$, $j \in \mathbb N$, and $(X_\infty, d_\infty, \tau_\infty)$
be compact timed-metric-spaces
converging in the intrinsic timed-Hausdorff
sense as in Definition \ref{defn:tau-K-dist},
\be 
d_{\tau-H}\Big((X_j,d_j,\tau_j),(X_\infty,d_\infty,\tau_\infty)\Big)\to 0,
\ee
If the  $(X_j,d_j, \tau_j)$, $j \in \mathbb N$,
are future-developed as in Definition~\ref{defn:FD}, then $\tau_\infty^{-1}(0) \neq \emptyset$, 
$\tau_\infty= d_\infty(\tau_\infty^{-1}(0), \,\,)$, and the sequence
$(X_j,\tau_j^{-1}(0))$ converges as metric pairs 
in the Gromov-Hausdorff sense to $(X_\infty,\tau_\infty^{-1}(0))$. 
\end{thm}

\begin{thm}[Part of Conjecture 6.5 in \cite{SakSor-Notions}]\label{conj-FD-to-tK}
Let $(X_j,d_j, \tau_j)$, $j \in \mathbb N$, and $(X_\infty, d_\infty, \tau_\infty)$
be compact future-developed timed-metric-spaces as in Definition~\ref{defn:FD}
with
\be 
d_{FD-HH}\Big((X_j,d_j, \tau_j),(X_\infty,d_\infty,\tau_\infty)\Big)\to 0,
\ee
then
\be 
d_{\tau-H}\Big((X_j,d_j,\tau_j),(X_\infty,d_\infty,\tau_\infty)\Big)\to 0.
\ee
\end{thm}

The structure of the paper is as follows. In Section~\ref{sec:back} we review basic notions and results on the Gromov–Hausdorff distance, followed by the convergence notions for timed-metric-spaces introduced by Sakovich–Sormani. In Section~\ref{sect:relations} we prove the results stated in the introduction.

\tableofcontents

\section{{\bf{Background}}}\label{sec:back}

We review basic notions and results on the Gromov–Hausdorff distance, followed by the convergence notions for timed-metric-spaces introduced by Sakovich–Sormani, closely following \cite{SakSor-Notions}. For further background on metric spaces, we refer the reader to \cite{BBI}.

\subsection{Metric spaces and distances between them}

Given a metric space $(Z,d_Z)$, the {\bf Hausdorff distance in $Z$} between subsets, $X,Y\subset Z$, is defined as
\be
d_H^Z(X,Y)=\inf
\left\{ \varepsilon>0:  \begin{array}{c} \forall x\in X, \exists y\in Y 
\textrm{ s.t. }
d_Z(x,y)<\varepsilon\\
\forall y\in Y, \exists x\in X
\textrm{ s.t. }
d_Z(x,y)<\varepsilon
\end{array}.
\right\}
\ee

For two compact metric spaces $(X, d_X)$ and $(Y,d_Y)$,
the {\bf Gromov-Hausdorff distance} (GH) between them is defined by
\be
d_{GH}((X,d_X),(Y,d_Y))=\inf d_H^Z(\varphi(X),\psi(Y)),
\ee
where the infimum is over all
 distance-preserving maps,
$\varphi:X\to Z$ and $\psi:Y \to Z$, into all possible metric spaces, $Z$.
\medskip

We now focus in a particular class of distance-preserving maps. For that, let  
\[
\ell^\infty=\{(s_1,s_2,\ldots)\,: s_i \in {\mathbb{R}}, \, d_{\ell^\infty}((s_1,s_2,\ldots),(0,0,\ldots))<\infty\}
\]
where   
\[
d_{\ell^\infty}((s_1,s_2,\ldots),(r_1,r_2,\ldots))
=\sup\{|s_i-r_i|\,:\, i\in {\mathbb N}\}.
\]

\begin{defn}\label{defn:Kuratowski}
Given a compact metric space, $(X,d)$, and a countable and dense collection of points $\mathcal N=\{x_1,x_2,\ldots\} \subset X$, the Fr\'echet map of $X$ given by $\mathcal N$,
\[
\kappa_X=\kappa_{X,\mathcal N}: (X,d)\to (\ell^\infty, d_{\ell^\infty}),
\]
is defined as
\be\label{eq:Kuratowski}
\kappa_X(x)=(d(x_1,x),d(x_2,x),\ldots).
\ee
\end{defn}

Fr\'echet maps are distance-preserving, c.f. \cite{SakSor-Notions}. 
Following the notation of Sakovich-Sormani  (see Section 2.6 in \cite{SakSor-Notions}), given  two compact metric spaces $(X,d_X)$ and $(Y,d_Y)$, let 
\be\label{defn:kappa-GH}
d_{\kappa-GH}((X,d_X),(Y,d_Y))
: =\inf d_H^{\ell^\infty}(\kappa_X(X),\kappa_Y(Y))
\ee
where the infimum is taken over all pairs of Fr\'echet maps $\kappa_X$
and $\kappa_Y$. With this notation, the following inequalities hold:
\begin{align}\label{eq:kappa-GH-comp}
d_{GH}\le d_{\kappa-GH}\le 2 d_{GH}.
\end{align}
The first inequality follows right away. The second follows from the next proposition, which is proven in \cite{SakSor-Notions}. 

\begin{prop}\label{prop:2GH}
Let $(X,d_X)$
and $(Y,d_Y)$ be two compact metric spaces.
If 
\be
d_{GH}(X,Y)<R
\ee
then there exist Fr\'echet maps
$\kappa_X:X\to \ell^\infty$
and 
$\kappa_Y:Y\to \ell^\infty$
such that 
\be\label{K-Haus}
d_H^{\ell^\infty}(\kappa_X(X),\kappa_Y(Y))
< 2R.
\ee
More explicitly, there exist
countable dense sets
\[
\{x_1,x_2,\ldots\}\subset X
\textrm{ and }
\{y_1,y_2,\ldots\}\subset Y
\]
such that the  Fr\'echet maps defined by them satisfy \eqref{K-Haus}.
\end{prop}

\bigskip

\subsection{Timed-metric-spaces and distances between them}\label{ssec:backtimed}

For this section recall the definition of timed-metric space, Definition \ref{defn:timed-metric-space}.

\begin{defn}\label{defn:tau-K}
Given a compact timed-metric-space, $(X,d,\tau)$,
and a countably dense collection of points,
${\mathcal N} \subset X$,
we define the
timed-Fr\'echet map, 
\[
\kappa_{\tau,X}=\kappa_{\tau,X,{\mathcal N}}: X \to [0,\tau_{max}]\times \ell^\infty\subset \ell^\infty,
\]
by $\kappa_{\tau,X}=(\tau,\kappa_{X})$, 
where $\kappa_{X}=\kappa_{X, \mathcal N}$ is as in Definition
\ref{defn:Kuratowski} and 
$\tau_{max}=\max_{x \in X}\tau(x)$.
\end{defn}

Timed-Fr\'echet maps are distance-preserving and restricting to them, one can define the intrinsic timed-Hausdorff distance.

\begin{defn}\label{defn:tau-K-dist}
The {\bf intrinsic 
timed-Hausdorff
distance}
between two compact {timed-metric-spaces}, 
$(X_1,d_1,\tau_1)$ and $(X_2,d_2,\tau_2)$, is defined as 
\be \label{eq:tK-GH-1}
d_{\tau-H}
\Big((X_1,d_1,\tau_1),(X_2,d_2,\tau_2)\Big)=
\inf\, d^{\ell^\infty}_H(
\kappa_{\tau_1,X_1}(X_1),
\kappa_{\tau_2,X_2}(X_2))
\ee
where the infimum is taken over all possible 
timed-Fr\'echet maps
$\kappa_{\tau_1,X_1}$ and
$\kappa_{\tau_2,X_2}$.
\end{defn}

\medskip

We now define two special cases of timed-metric-spaces and corresponding intrinsic distances between them. 

\begin{defn} \label{defn:BB}
A compact {\bf timed-metric-space  $(X,d, \tau)$ has a big bang} if 
$\tau^{-1}(0)$ 
consists of a single point, that we denote as
$p_{BB}$ and call it the big bang point, and the time function $\tau$ equals the
distance function to this point:
\be\label{eq:dist-BB}
\tau(p)= d(p,p_{BB})   \qquad \forall p\in X. 
\ee
In this case, we say that $(X,d, \tau)$ is a big bang timed-metric-space or a BB space.  
\end{defn}

Equivalently, we could define
$(X,d, \tau)$ to be a BB space if
there exists a point 
$p_{BB} \in X$, called the big bang point, 
such that the time function $\tau$ is given as in
\eqref{eq:dist-BB}.

\begin{defn}\label{defn:FD}
A {\bf future-developed timed-metric-space}\footnote{The current synthetic definition of a future-developed timed-metric-space is not yet fully satisfactory. In the smooth Lorentzian setting, 
$Y \subset X$ must be a hypersurface (satisfying suitable conditions). In particular, it cannot be a big bang space
and, thus, this definition precisely rules out this case.
},
$(X,d, \tau)$,  consists of a 
timed-metric-space $(X,d, \tau)$ and a non-empty set $Y \subset X$, of cardinality greater than or equal to 2, such that 
\[ 
Y=\tau^{-1}(0) \qquad \text{and} \qquad \tau(\cdot)= d(Y, \cdot).
\]
In this case, we also say that $(X,d, \tau)$ is a FD space.  
\end{defn}

\begin{defn} \label{defn:BB-GH}
The {\bf intrinsic $BB-GH$} distance between two compact
{\bf big bang timed-metric-spaces}
is defined as 
\be
d_{BB-GH}\left(\big(X_1,d_{1}, \tau_1\big),
\big(X_2,d_{2}, \tau_2 \big)\right)=
d_{pt-GH}\left(\big(X_1,d_{1}, p_{BB,1}\big),
\big(X_2,d_{2}, p_{BB,2}\big)\right),
\ee
where $p_{BB,i}=\tau_i^{-1}(0)$.
Here, 
\[
d_{pt-GH}\left(\big(X_1,d_{1}, p_{BB,1}\big) ,
\big(X_2,d_{2}, p_{BB,2}\big)\right)\]
is defined as the infimum 
over all distance-preserving maps $\varphi_i:X_i \to Z$  and metric spaces $Z$
of sums of the form:
\[
d_H^Z\big(\varphi_1(X_1),\varphi_2(X_2)\big)
+ d_Z(\varphi_1(p_{BB,1}),\varphi_2(p_{BB,2})).
\]
\end{defn}

\begin{defn}\label{defn:FD-HH}
The {\bf $FD-HH$ distance} between two compact future-developed space-times is defined as
\be \label{eq:FD-HH}
d_{FD-HH}\Big((X_1,d_1, \tau_1),(X_2,d_2,\tau_2)\Big)=
\ee
 
\be \label{eq:FD-HH-2}
=\inf \Bigg( d_H^Z\big(\varphi_1(X_1),\varphi_2(X_2)\big)
+ 
d_H^Z\big(\varphi_1(Y_1),\varphi_2(Y_2)\big)
\Bigg),
\ee
where the infimum is over all distance-preserving maps $\varphi_i: X_i \to Z$  and over all metric spaces $Z$.
\end{defn}

Finally, we recall that given compact metric spaces $(X_j,d_j)$
and subsets $Y_j \subset X_j$, $j \in \mathbb N \cup \{\infty\}$, it is said that $(X_j, Y_j)$, $j \in \mathbb N$,
converge as metric pairs in the Gromov-Hausdorff to  $(X_\infty,d_\infty)$ sense if 
\be
\inf \Bigg( d_H^Z\big(\varphi_j(X_j),\varphi_\infty(X_\infty)\big)
+ 
d_H^Z\big(\varphi_j(Y_j),\varphi_\infty(Y_\infty)\big)
\Bigg) \to 0 \qquad  \text{as }j \to \infty,
\ee
where the infimimum is taken over all metric  spaces $Z$ and distance-preserving maps $\varphi_j: X_j \to Z$. For additional details, see \cite{CGGGMS, Che-Gomez-pairs}.

\section{{\bf Relationships between the notions of convergence}}
\label{sect:relations}

Here we prove the results stated in the introduction.

\subsection{Proof of Theorem \ref{conj:tK-conv-to-others}}

\begin{lem}\label{lem:tK-conv-to-others}
If $(X_1,d_1,\tau_1)$ and $(X_2, d_2, \tau_2)$ are compact timed-metric-spaces such that their intrinsic timed-Hausdorff distance is bounded above,
\be 
d_{\tau-H}\Big((X_1,d_1,\tau_1),(X_2,d_2,\tau_2)\Big) < \varepsilon,
\ee
then their (timeless) Gromov-Hausdorff distance is also bounded above:
\begin{eqnarray}
d_{GH}\Big((X_1,d_1),(X_2,d_2)\Big) < \varepsilon. 
\end{eqnarray}
\end{lem}

\begin{proof}
By definition of intrinsic timed-Hausdorff distance, Definition \ref{defn:tau-K-dist},  there exist timed-Fr\'echet maps
$\varphi= \kappa_{\tau_1,X_1}$ and 
$\psi= \kappa_{\tau_2,X_2}$ such that setting $Z=\ell^\infty$, we have
\be\label{eq-est-a}
d^{Z}_H( \varphi (X_1), \psi(X_2)) < \varepsilon.
\ee
Denote by $F: Z \to \ell^\infty$ the map that deletes the first coordinate, i.e.
\[
F(x_1, x_2, x_3,\ldots)=(x_2, x_3,\ldots).
\]
Then, for the Fr\'echet maps $\kappa_{X_1}= F \circ \varphi$ and $ \kappa_{X_2}= F \circ \psi$ by \eqref{eq-est-a}, 
we have 
\[
d^{\ell^\infty}_H( \kappa_{X_1}(X_1), \kappa_{X_2}(X_2)) < \varepsilon.
\]
This implies, recalling the definition of Gromov-Hausdorff distance, that 
\[
d_{GH}
\Big((X_1,d_1),(X_2,d_2)\Big) <  \varepsilon,
\]
which is what we wanted to prove. 
\end{proof}

\bigskip

With the previous lemma at hand we establish our first main result. 

\bigskip

\begin{proof}[{\bf Proof of Theorem \ref{conj:tK-conv-to-others}}]
Assume 
\[
d_{\tau-H}\Big((X_j,d_j,\tau_j),(X_\infty,d_\infty, \tau_\infty)\Big)\to 0.
\]
Then, for all $\varepsilon >0$ there exists $N(\varepsilon) \in \N$ such that for all $j \geq N(\varepsilon)$, 
we have 
\[
d_{\tau-H}\Big((X_j,d_j, \tau_j),(X_\infty,d_\infty, \tau_\infty)\Big) < \varepsilon.
\]
Thus, by applying Lemma \ref{lem:tK-conv-to-others} we obtain
\[
d_{GH}
\Big((X_j,d_j),(X_\infty,d_\infty)\Big) <  \varepsilon,
\]
which implies \eqref{eq-timelessGH} and concludes the proof. 
\end{proof}

\subsection{Proof of first part of  Theorem \ref{conj:tK-conv-to-BB} and Theorem \ref{conj:tK-conv-FDm}}

In this subsection, we give the proofs of two results stated in the introduction, which are very similar.

\begin{proof}[{\bf Proof of first part of Theorem \ref{conj:tK-conv-to-BB}}]
Assume 
\[
d_{\tau-H}\Big((X_j,d_j, \tau_j),(X_\infty,d_\infty, \tau_\infty)\Big)\to 0.
\]
Thus, by definition, for all $\varepsilon >0$ there exists $N(\varepsilon) \in \N$ such that for all $j \geq N(\varepsilon)$, 
we have 
\[
d_{\tau-H}\Big((X_j,d_j, \tau_j),(X_\infty,d_\infty, \tau_\infty)\Big) < \varepsilon.
\]
By definition of intrinsic timed-Hausdorff distance, Definition \ref{defn:tau-K-dist},  there exist 
timed-Fr\'echet maps
$\varphi_j= \kappa_{\tau_j,X_j}$ and 
$\psi_j= \kappa_{\tau_\infty,X_\infty}$
such that for $j\geq N(\varepsilon)$ and setting $Z=\ell^\infty$, we have
\be\label{eq-est-b}
d^{Z}_H( \varphi_j(X_j), \psi_j(X_\infty)) < \varepsilon.
\ee
Note that $\varphi_j$ and 
$\psi_j$ also depend on $\varepsilon$, when necessary and to avoid confusion, we will add a superscript and denote them as $\varphi^\epsilon_j$ and 
$\psi_j^\epsilon$, respectively.

First we show that $\tau_\infty^{-1}(0) \neq \emptyset$.
Fix $j\geq N(\varepsilon)$. By hypotheses, $\tau_j^{-1}(0)$ consists of only one point, which we denote by $p_j$ for brevity. By \eqref{eq-est-b}  there exist $y^\epsilon_{j} \in X_\infty$ such that 
\be
d_Z(\varphi_j^\epsilon(p_{j}), \psi_j^\epsilon(y_{j}^\epsilon)) < \varepsilon.
\ee
 By the previous inequality, and since 
  $\varphi^\epsilon_j(p_j)-\psi_j^\epsilon(y_j^\epsilon)$ has as first entry
$ \tau_j(p_{j}) - \tau_\infty(y^\epsilon_{j})$, 
we see that
\be
\tau_\infty(y^\epsilon_{j}) = | \tau_j(p_{j}) - \tau_\infty(y^\epsilon_{j})|< \varepsilon.
\ee
Hence, taking a sequence of positive numbers $\varepsilon_j \to 0$, we can construct a subsequence of points, that we do not relabel, $y^{\epsilon_j}$ contained in $X_\infty$ such that $\tau_\infty(y^{\epsilon_j}) < \epsilon_j$. Since $X_\infty$ is compact, there exists a point $y \in X_\infty$ such that, up to a subsequence, $y^{\epsilon_j} \to y$, and by continuity of $\tau_\infty$, we get $\tau_\infty(y)=0$.

Now we show that $\tau_\infty^{-1}(0)$ consists of only one point. 
Assume that  $y_1, y_2 \in \tau_\infty^{-1}(0)$.
Fix $j\geq N(\varepsilon)$. By \eqref{eq-est-b}  there exist $x_{j1},x_{j2} \in X_j$ such that 
\be\label{eq-pointsxy} 
d_Z(\varphi_j(x_{ji}), \psi_j(y_{i})) < \varepsilon, \qquad i=1,2.
\ee
Since $y_1, y_2 \in \tau_\infty^{-1}(0)$, by the previous inequality and given that $\varphi_j(x_{ji})-\psi_j(y_i)$ has as first entry
$\tau_j(x_{ji}) - \tau_\infty(y_{i})$,
we see that
\be\label{eq-taus}
\tau_j(x_{ji}) = | \tau_j(x_{ji}) - \tau_\infty(y_{i})|< \varepsilon  \qquad i=1,2.
\ee
By hypotheses $d_j(x_{ji}, p_j)= \tau_j(x_{ji})$. Hence, by the triangle inequality, 
\eqref{eq-pointsxy} and \eqref{eq-taus}:
\begin{align*}
d_{\infty}(y_1, y_2) & = d_Z(\psi_j(y_{1}), \psi_j(y_{2})) \\
& \leq 
d_Z(\psi_j(y_{1}), \varphi_j(x_{j1}))+
d_Z(\varphi_j(x_{j1}), \varphi_j(p_{j}))\\
& \, + d_Z(\varphi_j(p_{j}), \varphi_j(x_{j2}))
+d_Z(\varphi_j(x_{j2}),\psi_j(y_{2}))\\
& < 4\varepsilon.
\end{align*}
Since $\varepsilon>0$ is arbitrary, this shows that $y_1=y_2$. 
Denote by $p_\infty$ the only point in  $\tau_\infty^{-1}(0)$. \bigskip

Now we show that $\tau_\infty= d_\infty(p_\infty, \,\, )$. Fix $j\geq N(\varepsilon)$.
First note that by \eqref{eq-pointsxy} and \eqref{eq-taus}, we have
\begin{align}\label{eq-dist-bbpoints}
d_Z(\psi_j(p_{\infty}),\varphi_j(p_{j}))
& \leq 
d_Z(\psi_j(p_{\infty}),\varphi_j(x_{ji}))
+
d_Z(\varphi_j(x_{ji}),\varphi_j(p_{j})) \notag \\
& =  d_Z( \psi_j(y_i),\varphi_j(x_{ji}))
+ \tau_j(x_{ji}) < 2\varepsilon.
\end{align}

Let $y \in X_\infty$. Take $x_y \in X_j$ such that 
\be\label{eq-yxy}
d_Z(\varphi_j(x_y), \psi_j(y)) < \varepsilon.
\ee
By considering the first coordinate function of $\varphi_j$ and $\psi_j$, we know that 
$|\tau_\infty(y) - \tau_j(x_y) | < \varepsilon$. Furthermore,  
\begin{align*}
 \tau_j(x_y) & =  d_Z(\varphi_j(x_y), \varphi_j(p_j)) \\
& \leq  d_Z(\varphi_j(x_y), \psi_j(y)) +   d_Z(\psi_j(y), \psi_j(p_\infty)) +d_Z(\psi_j(p_\infty), \varphi_j(p_j))   \\
& <   d_Z(\psi_j(y), \psi_j(p_\infty)) + 3 \varepsilon,
\end{align*}
where we used the triangle inequality, \eqref{eq-yxy} and \eqref{eq-dist-bbpoints}.  We can obtain a similar inequality and conclude
that 
\[
|\tau_\infty(y) -  d_Z(\psi_j(y), \psi_j(p_\infty))   | < 3\varepsilon.
\]
Since $\varepsilon>0$ is arbitrary this proves that $\tau_\infty=d_\infty(p_\infty, \,\,)$. Hence, $(X_\infty, d_\infty, \tau_\infty)$ is a BB space. 

\bigskip
We finally show BB-GH convergence. Denote by $F: \ell^\infty \to \ell^\infty$ the map given by  $F(x_1, x_2, x_3,\ldots)=(x_2, x_3,\ldots)$. 
For the Fr\'echet maps $\kappa_{X_j}= F \circ \varphi_j$ and $\kappa_{X_\infty}= F \circ \psi_j$ by \eqref{eq-est-b}
and  \eqref{eq-dist-bbpoints}, 
we have 
\be
d^{\ell^\infty}_H( \kappa_{X_j} (X_j), \kappa_{X_\infty}(X_\infty)) + d^{\ell^\infty} ( \kappa_{X_j} (p_j), \kappa_{X_\infty}(p_\infty)) <  3\varepsilon.
\ee 
Concluding that
\be
d_{BB-GH}\Big((X_j,d_j, \tau_j),(X_\infty,d_\infty,\tau_\infty)\Big) < 3 \varepsilon.
\ee
\end{proof}

\begin{proof}[ \bf{Proof of Theorem \ref{conj:tK-conv-FDm}}]
Assume 
\[
d_{\tau-H}\Big((X_j,d_j, \tau_j),(X_\infty,d_\infty, \tau_\infty)\Big)\to 0.
\]
Thus, by definition, for all $\varepsilon >0$ there exists $N(\varepsilon) \in \N$ such that for all $j \geq N(\varepsilon)$, 
we have 
\[
d_{\tau-H}\Big((X_j,d_j, \tau_j),(X_\infty,d_\infty, \tau_\infty)\Big) < \varepsilon.
\]
By definition of intrinsic timed-Hausdorff distance, Definition \ref{defn:tau-K-dist},  there exist 
timed-Fr\'echet maps
$\varphi_j= \kappa_{\tau_j,X_j}$ and 
$\psi_j= \kappa_{\tau_\infty,X_\infty}$ such that for $j\geq N(\varepsilon)$ and setting $Z=\ell^\infty$, we have
\be\label{eq-est-c}
d^{Z}_H( \varphi_j(X_j), \psi_j(X_\infty)) < \varepsilon.
\ee
Note that $\varphi_j$ and 
$\psi_j$ also depend on $\varepsilon$, when necessary and to avoid confusion, we will add a superscript and denote them as $\varphi^\epsilon_j$ and 
$\psi_j^\epsilon$, respectively. 

First we show that $\tau_\infty^{-1}(0) \neq \emptyset$. 
Fix $j\geq N(\varepsilon)$. By hypotheses, there exists 
$p_j \in \tau_j^{-1}(0)$. By \eqref{eq-est-c}  there exists $y^\epsilon_{j} \in X_\infty$ such that 
\be
d_Z(\varphi_j^\epsilon(p_{j}), \psi_j^\epsilon(y_{j}^\epsilon)) < \varepsilon.
\ee
 By the previous inequality, and since 
  $\varphi^\epsilon_j(p_j)-\psi_j^\epsilon(y_j^\epsilon)$ has as first entry
$ \tau_j(p_{j}) - \tau_\infty(y^\epsilon_{j})$, 
we see that
\be
\tau_\infty(y^\epsilon_{j}) = | \tau_j(p_{j}) - \tau_\infty(y^\epsilon_{j})|< \varepsilon.
\ee
Hence, taking a sequence of positive numbers $\varepsilon_j \to 0$, we can construct a subsequence of points, that we do not relabel, $y^{\epsilon_j}$ contained in $X_\infty$ such that $\tau_\infty(y^{\epsilon_j}) < \epsilon_j$. Since $X_\infty$ is compact, there exists a point $y \in X_\infty$ such that, up to a subsequence, $y^{\epsilon_j} \to y$, and by continuity of $\tau_\infty$, we get $\tau_\infty(y)=0$. 

\bigskip

We now show that $\tau_\infty= d_\infty(Y_\infty, \,\, )$. Fix $j \geq N(\varepsilon)$.
Let $y \in X_\infty$. By \eqref{eq-est-c} we can take $x_y \in X_j$ such that 
\be\label{eq-yxyc}
d_Z(\varphi_j(x_y), \psi_j(y)) < \varepsilon.
\ee
Since $X_j$ is a compact FD space, there exists $p \in Y_j$ such that 
\begin{equation}\label{eq-tauj-achieved}
    \tau_j(x_y)=d_j(Y_j,x_y)=d_j(p,x_y).
\end{equation}
Furthermore, there exists $y_p \in X_\infty$ such that 
\be\label{eq-pypc}
d_Z(\varphi_j(p), \psi_j(y_p)) < \varepsilon.
\ee

Hence, recalling that $\tau_j$ and $\tau_\infty$ are the 
first coordinate functions of $\varphi_j$ and $\psi_j$, respectively, we get 
\begin{align}
\tau_\infty(y_p) & =  |\tau_\infty(y_p)- \tau_j(p)| \\
& \leq  d_Z(\psi_j(y_p),\varphi_j(p)) < \varepsilon.
\end{align}
By \eqref{eq-tauj-achieved}, \eqref{eq-yxyc}, the triangle inequality, \eqref{eq-pypc} and \eqref{eq-tauj-achieved}, 
we get
\begin{align*}
\tau_\infty(y) & \leq  \tau_j(x_y) + d_Z( \psi_j(y), \varphi_j(x_y)) \\
& \leq  d_j(p,x_y) + \varepsilon\\
& \leq \{d_Z(\varphi_j(p), \psi_j(y_p)) +
d_Z(\psi_j(y_p), \psi_j(y)) + d_Z( \psi_j(y), \varphi_j(x_y))\} + \varepsilon\\
& \leq  d_\infty(y_p,y) + 3 \varepsilon.
\end{align*}
Furthermore, 
\begin{align*}
-\tau_\infty(y) & \leq - \tau_j(x_y) + d_Z( \psi_j(y), \varphi_j(x_y)) \\
& \leq  -d_j(p,x_y) + \{ d_Z( \psi_j(y), \psi_j(y_p)) +  d_Z(\psi_j(y_p),\varphi_j(p)) +
d_Z(\varphi_j(p), \psi_j(x_y))\} \\
& \leq  d_\infty(y_p,y) + \varepsilon.
\end{align*}
Hence, 
\[
|\tau_\infty(y) - d_\infty(y_p,y) | \leq 3 \varepsilon.
\]
Since $\varepsilon>0$ is arbitrary, one can construct a sequence of points $y_i \in X_\infty$
such that $\tau_\infty(y_i)< \varepsilon_i$ and 
\[
|\tau_\infty(y) - d_\infty(y_i,y) | \leq 3 \varepsilon_i.
\]
for $\varepsilon_i \searrow 0$.  Since $X_\infty$ is compact, up to a subsequence, $y_i$ converges to $q \in X_{\infty}$ and since $\tau_\infty$ is continuous, $q \in Y_\infty$. Hence, $\tau_\infty(y)= d_\infty(y,q)\geq d_\infty(Y,y)$.
Since $\tau_\infty$ is $1$-Lipschitz, the other inequality follows, 
\[
\tau_{\infty}(y)= \min_{z \in Y_\infty} |\tau_{\infty}(y)- \tau_{\infty}(z)| \leq \min_{z \in Y_\infty}  d_{\infty}(y,z) = d_\infty(Y,y).
\]
Hence,  $\tau_\infty= d_\infty(Y_\infty, \,\, )$.

\bigskip
Denote by $F: \ell^\infty \to \ell^\infty$ the map given by  $F(x_1, x_2, x_3,\ldots)=(x_2, x_3,\ldots)$. For $j \geq N(\varepsilon)$, and the Fr\'echet maps $\kappa_{X_j}= F \circ \varphi_j$ and $\kappa_{X_\infty}= F \circ \psi_j$, by \eqref{eq-est-c}, 
we have 
\be
d^{\ell^\infty}_H( \kappa_{X_j} (X_j), \kappa_{X_\infty}(X_\infty)) <  \varepsilon.
\ee 
Finally, to see that 
\be
d^{\ell^\infty}_H( \kappa_{X_j} (Y_j), \kappa_{X_\infty}(Y_\infty)) <  2\varepsilon,
\ee 
we proceed as follows. Let $p \in Y_j$. By \eqref{eq-est-c}, we can take $y_p \in X_\infty$ such that 
\be
d_Z(\varphi_j(p), \psi_j(y_p)) < \varepsilon.
\ee
Hence, $\tau_\infty(y_p)=|\tau_\infty(y_p) -
\tau_j(p)|  < \varepsilon$. Since 
$X_\infty$ is a compact FD space, there exists $q \in Y_\infty$ such that 
\[
\tau_\infty(y_p)=d_\infty(y_p,q).
\]
Furthermore, there exists $y_p \in X_\infty$ such that 
\be
d_Z(\varphi_j(p), \psi_j(y_p)) < \varepsilon.
\ee
By the triangle inequality, we get
\be
d_Z(\varphi_j(p), \psi_j(q)) < 2\varepsilon.
\ee
Similarly, for each $q \in Y_\infty$ there exists $p \in Y_j$ such that the previous inequality holds. This implies the required estimate and concludes the proof. 
\end{proof}

\subsection{Proof of second part of Theorem \ref{conj:tK-conv-to-BB} and Proof of Theorem \ref{conj-FD-to-tK}}

In this subsection, we give the proofs of two results stated in the introduction, which are very similar.

\begin{lem}\label{lem:2TkGH}
Let $(X_j,d_j, \tau_j)$
and $(X_\infty,d_\infty, \tau_\infty)$ be two compact big bang timed-metric-spaces. Assume that 
\be 
d_{BB-GH}\Big((X_j,d_j,\tau_j),(X_\infty,d_\infty,\tau_\infty)\Big) < \varepsilon.
\ee
Then
\be 
d_{\tau-H}\Big((X_j,d_j,\tau_j),(X_\infty,d_\infty,\tau_\infty)\Big) < 2\varepsilon.
\ee
More explicitly, there exist
countable dense sets $\mathcal N_{X_j}$
and $\mathcal N_{X_\infty}$
such that the  timed-Fr\'echet maps $\kappa_{\tau_j,X_j}$
and $\kappa_{\tau_\infty,X_\infty}$ defined by them satisfy
\be\label{KT-Haus}
d_H^{\ell^\infty}(\kappa_{\tau_j,X_j}(X_j),\kappa_{\tau_\infty, X_\infty}(X_\infty))
< 2\varepsilon.
\ee
\end{lem}

The proof is very similar to the proof of Proposition \ref{prop:2GH}, but we give all the details for completeness. 

\begin{proof}
Assume that 
\be 
d_{BB-GH}\Big((X_j,d_j,\tau_j),(X_\infty,d_\infty,\tau_\infty)\Big) < \varepsilon.
\ee
By definition of $BB-GH$ distance, there exist a compact metric space $Z$ and
isometric embeddings $\varphi_j: X_j \to Z$  and $\varphi_\infty: X_\infty \to Z$
such that 
\be\label{eq-BBeps}
d_H^Z\big(\varphi_j(X_j),\varphi_\infty(X_\infty)\big)
+ d_Z(\varphi_j(p_{BB,j}),\varphi_\infty(p_{BB,\infty})) < \varepsilon.
\ee
So for all $x\in X_j$ there exists $y_x\in X_\infty$
such that
\be\label{eq:y_x-Haus}
d_Z(\varphi_j(x),\varphi_\infty(y_x))<\varepsilon,
\ee
and for all $y\in X_\infty$ there exists $x_y\in X_j$
such that
\be\label{eq:x_y-Haus}
d_Z(\varphi_j(x_y),\varphi_\infty(y))<\varepsilon.
\ee
Let $\{x'_1,x'_2,\dots\}$ and 
$\{y'_1,y'_2,\ldots\}$
be two countably dense collection of points in $X_j$ and $X_\infty$, respectively. Applying \eqref{eq:y_x-Haus} and 
 \eqref{eq:x_y-Haus}, we obtain larger countably and dense collections of points,
\[
\mathcal N_j=\{x_1,x_2,x_3\ldots\}=\{x_1',x_{y_1'},x_2',x_{y_2'},
x_3',x_{y_3'},\ldots\}\subset X_j
\]
and
\[
\mathcal N_\infty=\{y_1,y_2,y_3,\ldots\}=
\{y_{x_1'},y_1',y_{x_2'}, y_2',
y_{x_3'},y_3',\ldots\}\subset X_\infty.
\]
We then define timed-Fr\'echet maps,
$\kappa_{\tau_{X_j}\,X_j}=(\tau_j, \kappa_{X_j})$ and 
$\kappa_{\tau_\infty\,X_\infty}=(\tau_\infty, \kappa_{X_\infty})$,
using these collections.

Given $x \in X_j$ and $y_x \in X_\infty$
that satisfy \eqref{eq:y_x-Haus}, we claim that   
\be\label{eqPrim}
d_{\ell^\infty}(\kappa_{\tau_j\, X_j}(x),\kappa_{\tau_\infty\, X_\infty}(y_x))
<2\varepsilon. 
\ee
To prove the claim, recall that 
\begin{align*}
\kappa_{\tau_j X_j}(x)= & (\tau_{j}(x), d_j(x_1',x),d_j(x_{y_1'},x),d_j(x_2',x),
d_j(x_{y_2'},x),\ldots)\\
\kappa_{\tau_{\infty  X_\infty}}(y_x) = &
(\tau_\infty(y_x), d_\infty(y_{x_1'},y_x),d_\infty(y_1',y_x),
d_\infty(y_{x_2'},y_x), d_\infty(y_2',y_x),\ldots).
\end{align*}
Then by the fact that $\varphi_j$ and 
$\varphi_\infty$ are distance-preserving,
the triangle inequality,  \eqref{eq:y_x-Haus} and the second term to the left of \eqref{eq-BBeps}, we have
\begin{align}\label{eq-tauMaps}
|\tau_j(x) - \tau_\infty(y_x)| & = |d_j(x, p_{BBj}) - d_\infty(y_x, p_{BB\infty})| \\
& \leq d_Z( \varphi_j(x), \varphi_\infty(y_x)) +  d_Z(\varphi_j(p_{BBj}), \varphi_\infty(p_{BB\infty})) \notag\\
& < 2\varepsilon. \notag
\end{align}
Similarly, 
for any $i \in \mathbb N$, 
\begin{align*}
|d_j(x_i',x)-d_\infty(y_{x_i'},y_x)|
& =
|d_Z(\varphi_j(x_i'),\varphi_j(x))-
d_Z(\varphi_\infty(y_{x_i'}),\varphi_\infty(y_x))|\\
& \le 
d_Z(\varphi_j(x_i'),\varphi_\infty(y_{x_i'}))
+ d_Z(\varphi_j(x),\varphi_\infty(y_x)) \\
& <  2\varepsilon, \\
|d_j(x_{y_i'},x)-d_\infty(y_i',y_x)| & <2\varepsilon.
\end{align*}
Hence, \eqref{eqPrim} holds.  Analogously, one can show that  
for any $y \in X_\infty$ there exists $x_y \in X_j$
such that
\be\label{eqPrim2}
d_{\ell^\infty}(\kappa_{\tau_j \, X_j}(x_y),\kappa_{\tau_\infty \, X_\infty}(y))
<2\varepsilon.
\ee
From \eqref{eqPrim} and \eqref{eqPrim2} we conclude that \eqref{KT-Haus} holds.
\end{proof}

\bigskip 

\begin{proof}[{\bf Proof of second part of Theorem \ref{conj:tK-conv-to-BB}}]
Assume that 
\be 
d_{BB-GH}\Big((X_j,d_j,\tau_j),(X_\infty,d_\infty,\tau_\infty)\Big)\to 0,
\ee
Then, for all $\varepsilon >0$ there exists $N(\varepsilon) \in \N$ such that for all $j \geq N(\varepsilon)$, 
we have 
\be 
d_{BB-GH}\Big((X_j,d_j,\tau_j),(X_\infty,d_\infty,\tau_\infty)\Big) < \varepsilon.
\ee
Now by Lemma \ref{lem:2TkGH}, for all $j \geq N(\varepsilon)$ we have
\be 
d_{\tau-H}\Big((X_j,d_j,\tau_j),(X_\infty,d_\infty,\tau_\infty)\Big) < 2\varepsilon,
\ee
which concludes the proof of the theorem. 
\end{proof}

\bigskip

\begin{lem}\label{prop:FD-tK}
Let $(X_j,d_j, \tau_j)$
and $(X_\infty,d_\infty, \tau_\infty)$ be two compact future-developed timed-metric-spaces. Assume that  
\be
d_{FD-HH}\Big((X_j,d_j, \tau_j),(X_\infty,d_\infty, \tau_\infty)\Big) < \varepsilon,
\ee
then
\be 
d_{\tau-H}\Big((X_j,d_j,\tau_j),(X_\infty,d_\infty,\tau_\infty)\Big) < 2\varepsilon.
\ee
More explicitly, there exist
countable dense sets $\mathcal N_{X_j}$
and $\mathcal N_{X_\infty}$
such that the  timed-Fr\'echet maps $\kappa_{\tau_j,X_j}$
and $\kappa_{\tau_\infty,X_\infty}$ defined by them satisfy
\be\label{KT-HausF}
d_H^{\ell^\infty}(\kappa_{\tau_j,X_j}(X_j),\kappa_{\tau_\infty, X_\infty}(X_\infty))
< 2\varepsilon.
\ee
\end{lem}

\begin{proof}
The proof is very similar to the
one of Lemma \ref{lem:2TkGH}. So, we use the same notation as in that proof and only focus in the parts where they differ.

By definition of $FF-DH$, there exist a compact metric space $Z$ and
isometric embeddings, $\varphi_j:X_j \to Z$  and $\varphi_\infty:X_\infty \to Z$,
such that 
\be\label{eq-HFFDDest}
d_H^Z\big(\varphi_j(X_j),\varphi_\infty(X_\infty)\big)
+ d_H^Z\big(\varphi_j(Y_j),\varphi_\infty(Y_\infty)\big) < \varepsilon.
\ee
\bigskip
Construct timed-Fr\'echet maps, $\kappa_{\tau_j\,X_j}=(\tau_j, \kappa_{X_j})$ and 
$\kappa_{\tau_\infty\,X_\infty}=(\tau_\infty, \kappa_{X_\infty})$, 
as in the proof of Lemma \ref{lem:2TkGH}. Given $x \in X_j$ and $y_x \in X_\infty$
that satisfy
\[
d_Z(\varphi_j(x),\varphi_\infty(y_x))<\varepsilon,
\]
we claim that   
\[
d_{\ell^\infty}(\kappa_{\tau_j\, X_j}(x),\kappa_{\tau_\infty\, X_\infty}(y_x))
<2\varepsilon. 
\]
The proof of this claim proceeds as in the proof of Lemma \ref{lem:2TkGH}, except 
for \eqref{eq-tauMaps}. For this, note that $Y_j=\tau_j^{-1}(0) \subset X_j$ is nonempty and compact, so there exists 
$p_x \in Y_j$  such that 
\[
\tau_j(x) = d_j(Y,x)=d_{j}(x, p_{x}).
\]
By the second term to the left of \eqref{eq-HFFDDest} there exist 
$y_{p_x} \in Y_\infty$ such that
\[
d_Z(\varphi_j(p_x), \varphi_\infty(y_{p_x})) < \varepsilon.
\]
Hence,
\begin{align*}
\tau_\infty(y_x) - \tau_j(x)    & =   \tau_\infty(y_x)  - d_j(x, p_{x}) \\
 & \leq  d_\infty(y_x, y_{p_x}) -  d_j(x, p_{x}) \\ 
& \leq   d_Z(\varphi_j(x), \varphi_\infty(y_x)) + d_Z(\varphi_j(p_x), \varphi_\infty(y_{p_x})) < 2 \varepsilon,
\end{align*}
where we used the fact that $\tau_\infty = d(Y_\infty, \,\,)$. 
In a similar way, we get the other inequality and conclude that 
\begin{equation*}
|\tau_j(x) - \tau_\infty(y_x)| < 2\varepsilon. 
\end{equation*}
 This concludes the proof of the claim. 
The remaining part of the proof proceeds as in the proof of the Lemma \ref{lem:2TkGH}. 
\end{proof}

\bigskip 

\begin{proof}[\bf{Proof of Theorem \ref{conj-FD-to-tK}}]
This follows applying Lemma \ref{prop:FD-tK}.
\end{proof}

\bibliographystyle{alpha}
\bibliography{CPS-bib}
\end{document}